\documentclass[12pt,a4paper,twoside]{article}
\usepackage{amsmath,amsthm,amssymb,amsfonts,amsbsy}

\usepackage{color, colortbl} 
\usepackage{float} 
\usepackage[margin=2.5cm]{geometry} 
\usepackage{cite}
\usepackage{authblk}
\usepackage[utf8]{inputenc}
\usepackage{pgf,tikz}
\usetikzlibrary{arrows}

\newtheorem{thm}{Theorem}[section]
\newtheorem{cor}[thm]{Corollary}
\newtheorem{con}[thm]{Conjecture}
\newtheorem{lem}[thm]{Lemma}

\numberwithin{equation}{section}

\newcommand{\beq}{\begin{eqnarray}}
\newcommand{\eeq}{\end{eqnarray}}

\newcommand{\beqs}{\begin{eqnarray*}}
\newcommand{\eeqs}{\end{eqnarray*}}

\allowdisplaybreaks

\usepackage{mathtools}

\title{\bf  On a Conjecture about Degree Deviation Measure of Graphs}

\author{ \bf Ali Ghalavand\thanks{Corresponding author (alighalavand@grad.kashanu.ac.ir)}~ and  Ali Reza Ashrafi
}

\affil{ \normalsize
    { \it Department of Pure Mathematics, Faculty of Mathematical Sciences, University of Kashan, Kashan 87317--53153, I. R. Iran}}

\begin{document}

\maketitle
\begin{abstract}
Let $G$ be an $n-$vertex graph with $m$ edges. The degree deviation measure of $G$  is defined as
$s(G)$ $=$ $\sum_{v\in V(G)}|deg_G(v)- \frac{2m}{n}|,$ where  $n$ and $m$ are   the number of vertices and edges of $G$, respectively. The aim of this paper is to prove the Conjecture 4.2 of [J. A. de Oliveira, C. S. Oliveira, C. Justel and N. M. Maia de Abreu, Measures of irregularity of graphs, \textit{Pesq. Oper.} \textbf{33} (3) (2013) 383--398]. The degree deviation measure of chemical graphs under some conditions on the cyclomatic number is also computed.

\vskip 3mm


\vskip 3mm


\end{abstract}

\bigskip


\section{Introduction}
Throughout this paper all graphs are assumed to be simple and undirected. If $G$ is such a graph, then the vertex and edge sets of $G$ are denoted by $V(G)$ and $E(G)$, respectively. The degree of a vertex $v$ is denoted by $deg_G(v)$ (or $deg(v)$ for short),
$N[v, G]$ is the set of all vertices adjacent to $v$  and $H(n)$ denotes the set of all connected $n-$vertex graphs.

Let $k$ and $n$ be integer numbers such that $0 \leq k\leq n$. A graph $S(n, k)$ is a split
graph if there is a partition of its vertex set into a clique of order $k$ and a stable set of order $n -k$.
A complete split graph, $CS(n, k)$, is a split graph such that each vertex of the clique is adjacent
to each vertex of the stable set \cite{3}.

A graph in which all vertices have the same degree is said to be \textit{regular}. Suppose $\alpha$ is a function from the set of all graphs into non-negative integers. If for each regular graph $G$, $\alpha(G) = 0$, then the function $\alpha$ is called a \textit{measure of regularity}.

With the best of our knowledge the first irregularity measure was proposed by Albertson \cite{1}. He defined the irregularity of a graph $G$ by $irr(G) = \sum_{xy \in E(G)}|deg(x) - deg(y)|$ and determined the maximum irregularity of various classes of graphs. As a consequence of his results, the irregularity of an $n-$vertex graph is less that $\frac{4n^3}{27}$, and the bound is tight. Abdo et al. \cite{2} was proposed a generalization of irregularity measure of Albertson which is called the \textit{total irregularity}. It is defined as $irr_t(G)$ $=$ $\frac{1}{2}\sum_{u,v\in V(G)} |deg(u)-deg(v)|$. They obtained all graphs with maximal total irregularity and proved that among all trees of the same order the star has the maximal total irregularity.

Following Nikiforov \cite{4}, the \textit{degree deviation measure} of $G$  is defined as
$$s(G) = \sum_{v\in V(G)}\left| deg_G(v)- \frac{2m}{n}\right|,$$ where  $n$ and $m$ are   the number of vertices and edges of $G$, respectively. Let $G$ be a graph with maximum eigenvalue $\mu(G)$. The well-known result of Euler which states that $\sum_{v\in V(G)}deg_G(v)=2|E(G)|$ implies that the degree deviation measure $s(G)$ is a measure of irregularity.  Nikiforov  proved that $$\frac{s^2(G)}{2n^2\sqrt{2m}} \leq \mu(G)-\frac{2m}{n} \leq \sqrt{s(G)}$$ and these inequalities are tight up to a constant factor.

de Oliveira et al. \cite{3} investigated four distinct graph invariants used to measure the irregularity of a graph and proved the following theorem:

\begin{thm} \label{t1} Let $k \in N$ and $0 \leq k \leq n$. If $G = CS(n, k)$ is a complete split graph then
$s(G) =\frac{n}{2}k(n - k)(n - 1 - k)$. Besides, among all complete split graphs, the most irregular one
by $s(G)$ has to attend the following conditions on $k$:
$$k=\left\{\begin{array}{ll}
\frac{n}{3} & 3 \mid n\\
\frac{n-1}{3} & 3 \mid n - 1\\
\frac{n-2}{3} ~and ~\frac{n+1}{3} & 3 \mid n - 2
\end{array}\right..$$
\end{thm}

They conjectured that

\begin{con} \label{con} Let $H(n)$ be the set of all connected graphs $G$ with $n$ vertices. Then $\max_{G\in H(n)}
s(G) = s(CS(n, k))$, where
$$k=\left\{\begin{array}{ll}
\frac{n}{3} & 3 \mid n\\
\frac{n-1}{3} & 3 \mid n - 1\\
\frac{n-2}{3}~ and ~\frac{n+1}{3} & 3 \mid n - 2
\end{array}\right..$$
\end{con}

The aim of this paper is to prove Conjecture \ref{con}.

\section{Proof of the Conjecture}

In this section, we will present a proof for Conjecture \ref{con}. To do this, we need some notations as follows:
\begin{eqnarray*}
V^\downarrow(G)&=&\{v \in V(G) \mid deg_G(v)\leq \frac{2m}{n}\},\\
V^\uparrow(G)&=&\{v \in V(G) \mid deg_G(v) > \frac{2m}{n}\},\\
E^\downarrow(G)&=&\{uv\in E(G) \mid \{u,v\}\subseteq  V^\downarrow(G)\},\\
\overline{E}^\uparrow(G)&=&\{\{u,v\}\subseteq V^\uparrow(G) \mid uv\not\in E(G)\}.
\end{eqnarray*}

The well-known result of Euler which states that $\sum_{v\in V(G)}deg_G(v)=2|E(G)|$ lead us to the following useful lemma:

\begin{lem}
Let $F(n,m)$ denote the family of all connected graphs with $n$ vertices and $m$ edges. Then,
\begin{enumerate}
\item $\{G\in F(n,m) \mid  V^\downarrow(G)=V(G)\}=\{G\in F(n,m) \mid G \ is \ \frac{2m}{n}-regular\},$

\item $\{G\in F(n,m) \mid  V^\uparrow(G)=V(G)\}=\emptyset.$
\end{enumerate}
\end{lem}

\begin{lem}\label{l0}
Let $G$ be a connected  $n-$vertex irregular graph, $n\geq3$, and $e=uv\in E^\downarrow(G)$ is not a cut edge of $G$. If $G^- = G -e$ then $s(G)< s(G^-)$.
\end{lem}

\begin{proof}
Let  $V_1^\downarrow(G)=\{v\in  V^\downarrow(G): deg_G(v)>\frac{2m-2}{n}\}$. By definition,
\begin{eqnarray*}
s(G^-)-s(G)&=&\frac{2m-2}{n}- deg_G(u)+1+ \frac{2m-2}{n}- deg_G(v)+1\\
&-&[\frac{2m}{n}- deg_G(u)+ \frac{2m}{n}- deg_G(v)]
+\sum_{w\in V_1^\downarrow(G)}2\left[deg_G(w)-\frac{2m-1}{n}\right]\\
&-&\sum_{w\in  V^\downarrow(G)\setminus  (V_1^\downarrow(G)\cup\{u,v\})}\frac{2}{n}
+\sum_{v \in  V^\uparrow(G)}\frac{2}{n}\\
&\geq&2-\frac{2}{n}\left[\left|  V^\downarrow(G)\setminus V_1^\downarrow(G) \right| -\left|V^\uparrow(G)\right|\right]>0,
\end{eqnarray*}
proving the lemma.
\end{proof}

\begin{lem}\label{l1}
Let $G$ be a connected irregular $n-$vertex graph, $n\geq3$, and $\{w,z\}\in \overline{E}^\uparrow(G)$. If
${G^+} = G + f$, $f=wz$, then $s(G)<s(G^+)$.
\end{lem}

\begin{proof}
Let  $V_1^\uparrow(G)=\{v\in  V^\uparrow(G): deg_G(v)\leq\frac{2m+2}{n}\}$. By definition,
\begin{eqnarray*}
s(G^+)-s(G)&=&deg_G(u)+1-\frac{2m+2}{n}+ deg_G(v)+1-\frac{2m+2}{n} \\
&-&[deg_G(u)-\frac{2m}{n}+deg_G(v)- \frac{2m}{n} ]
+\sum_{w\in V_1^\uparrow(G)}2\left[\frac{2m+1}{n}-deg_G(w)\right]\\
&+&\sum_{w\in V^\downarrow(G)}\frac{2}{n}
-\sum_{w\in  V^\uparrow(G)\setminus  (V_1^\uparrow(G)\cup\{u,v\})}\frac{2}{n}\\
&\geq&2-\frac{2}{n}\left[\left|  V^\uparrow(G)\setminus V_1^\uparrow(G) \right| -\left|V^\downarrow(G)\right|\right]>0,
\end{eqnarray*}
proving the lemma.
\end{proof}

Suppose $G$  is a connected irregular graph with a cut edge $e=uv\in E^\downarrow(G)$. It is clear that there exists a vertex $w\in V^\uparrow(G)$ such that at least one of the graphs $G-uv+uw$ and $G-uv+vw$ is connected. Without loss of generality, we assume that $G-uv+uw$ is connected. Then the edge $uw$ is called a \textit{connectedness factor} of $G - e$ with respect to $V^\downarrow(G)$ and $V^\uparrow(G)$.

\begin{lem}\label{l2}
Let $G$ be a connected irregular $n-$vertex graph, $n\geq3$, and $e=uv\in E^\downarrow(G)$ is an cut edge of $G$. If
$w\in V^\uparrow(G)$,  $G^*=G-uv+uw$ and $uw$ is a connectedness factor of $G-uv$ with respect to $V^\downarrow(G)$ and
$V^\uparrow(G)$, then $s(G)<s(G^*)$.
\end{lem}

\begin{proof}
By definition,
\begin{eqnarray*}
s(G^*)-s(G)&=&\frac{2m}{n}-(deg_G(v)-1)+deg_G(w)+1-\frac{2m}{n}\\
&-&\left[\frac{2m}{n}-deg_G(v)+deg_G(w)-\frac{2m}{n}\right]\\
&=&2,
\end{eqnarray*}
as desired.
\end{proof}

\begin{lem}\label{l3}
Suppose $G$ is a connected irregular  $n-$vertex graph such that $V^\uparrow(G)=\{u_1,\ldots,u_k\}$ is a clique, $V^\downarrow(G)=\{v_1,\ldots,v_{n-k}\}$ is a stable set and $\sum_{i=1}^{n-k}deg_G(v_i)<|V^\downarrow(G)|k$. Let $N[G,v_i]=\{u_{i_1},\ldots,u_{i_{deg_G(v_i)}}\}$, for $1\leq i \leq n-k$.
 \begin{enumerate}
 \item If $|V^\downarrow(G)|>|V^\uparrow(G)|$ and  $G^\ddagger=G+\{v_iu_{j}: 1\leq i\leq n-k~and~u_j\in V^\uparrow(G)\setminus{N}[v,G]\}$, then $s(G)<s(G^\ddagger)$.
 \item If   $|V^\downarrow(G)|\leq|V^\uparrow(G)|$ and  $G^\intercal=G-\{v_iu_{i_j}: 1\leq i\leq n-k~and~2\leq j\leq deg_G(v_i) \}$, then $s(G)\leq s(G^\intercal)$.
 \end{enumerate}
\end{lem}

\begin{proof}
Let $\alpha=|\{v_iu_{j}: 1\leq i\leq n-k~and~u_j\in V^\uparrow(G)\setminus{N}[v,G]\} |$ and .
For prove (1) by definition,
\begin{eqnarray*}
s(G^\ddagger)-s(G)&=&k(k-1)+(n-k)k-\frac{2m+\alpha}{n}|V^\uparrow(G)|
+\frac{2m+\alpha}{n}|V^\downarrow(G)|-(n-k)k\\
&-&\left[k(k-1)+(n-k)k-\alpha-\frac{2m}{n}|V^\uparrow(G)|+\frac{2m}{n}|V^\downarrow(G)|-(n-k)k+\alpha   \right]\\
&=&\frac{\alpha}{n}|V^\downarrow(G)|-\frac{\alpha}{n}|V^\uparrow(G)|>0,
\end{eqnarray*}
as desired. For Prove (2) by  definition and similar with last proof, $s(G^\intercal)-s(G)=\frac{\alpha}{n}|V^\uparrow(G)|- \frac{\alpha}{n}|V^\downarrow(G)|\geq0$, as desired.

\end{proof}

Let $S^1(n,k)$ be a split graph such that all vertices in its stable set are of degree one. Define:
$$H(n,1):=\{G\in H(n):for~ one~ 1\leq k\leq n-1, G\cong S^1(n,k)\},$$
$$CH(n):=\{G\in H(n):for~ one~ 1\leq k\leq n-2, G\cong CS(n,k)\}.$$

\begin{lem}\label{l4}
Let $k\in N$ and $1< k < n-1$. Then $\max_{G\in H(n,1)}
s(G) =s(S^1(n,k))$, where
$$k=\left\{\begin{array}{ll}
\frac{2}{3}n & 3 \mid n\\
\frac{2}{3}(n-1) & 3 \mid n - 1\\
1~ and~ 2 & n=5\\
 \frac{2}{3}(n+1) & n\neq5~ and~ 3 \mid n-2\\
\end{array}\right..$$
\end{lem}

\begin{proof}
By definition $s(S^1(n,k))=\frac{2}{n}(3k-k^2-n)(n-k)$. For a given $n$, define $g(k)=\frac{2}{n}(3k-k^2-n)(n-k)$. By a simple calculation
the maximal value of $g (k)$ is obtained by $k=\frac{2}{3}n$. Since $k$ is an integer, we have to determine $\lceil k\rceil$ and  $\lfloor k \rfloor$ and then comparing $g(\lceil k\rceil)$ and  $g(\lfloor k \rfloor)$ in all previous cases give the our result.

\end{proof}

\begin{lem}\label{l5}
Let $k\in N$ and $1< k < n-1$. Then $\max_{G\in H(n,1)}
s(G) < \max_{G\in CH(n)}s(G)$.
\end{lem}

\begin{proof}
Let $\max_{G\in CH(n)}s(G)=\lambda$ and $\max_{G\in H(n,1)}
s(G)=\mu$.
By Theorem \ref{t1} and Lemma \ref{l4},
$$\lambda-\mu=\left\{\begin{array}{ll}
\frac{2}{9}n & 3 \mid n\\
\frac{2}{9}\frac{(n+2)^2-18}{n} & 3 \mid n - 1\\
 \frac{2}{9}\frac{(n+4)(n-2)}{n} & 3 \mid n-2
\end{array}\right.,$$
as desired.
\end{proof}


We are now ready to prove Conjecture \ref{con}.

\begin{thm} Let $H(n)$ be the set of all connected graphs $G$ with $n$ vertices. Then, $\max_{G\in H(n)}
s(G) = s(CS(n, k))$, where \
$$k=\left\{\begin{array}{ll}
\frac{n}{3} & 3 \mid n\\
\frac{n-1}{3} & 3 \mid n - 1\\
\frac{n-2}{3}~ and~ \frac{n+1}{3} & 3 \mid n-2
\end{array}\right..$$
\end{thm}

\begin{proof}
Suppose $G$ is not a regular graph. Then by  repeated applications of Lemmas \ref{l0}, \ref{l1} and \ref{l2}, we obtain a connected  split graph $H=S( n, k)$ such that $s(G)\leq s(H)$. Now by  repeated applications of Lemmas \ref{l3} on graph $H$, we obtain a $F=S^1( n, k)$ or $F=CS( n, k)$
such that $s(G)\leq s(H)\leq s(F)$ with equality if and only if $G\cong S^1( n, k)$ or $G\cong CS( n, k)$. The proof follows from Theorem \ref{t1} and
Lemmas \ref{l4}, \ref{l5}.
\end{proof}


\section{Connected Chemical Graphs}
The aim of this section is to continue the interesting paper \cite{15}. We will compute the degree deviation measure of chemical graphs under some conditions on the cyclomatic number.

Suppose  $n_i = n_i(G)$ is the number of vertices of degree $i$ in a graph $G$. It can be easily seen  that  $\sum_{i=1}^{\Delta(G)} n_i = |V(G)|$. If the graph $G$ has exactly $ n$ vertices, $m$ edges and $k$ components, then $c = m - n + k$ is called the \textit{cyclomatic number} of $G$. A chemical graph is a graph with a maximum degree of 4.
A connected chemical graphs with exactly $n$ vertices and cyclomatic number $c$  is called $(n,c)$-chemical graph.

\begin{lem}\label{lemm0}
\textrm{\rm (}See \cite{2-1}\textrm{\rm )} Let $G$ be an $(n,c)$-chemical graph.  Then
$$
n_{1}(G)=2-2c +n_3+2n_4 \hspace{5mm} \mbox{and} \hspace{5mm}
n_{2}(G) =2c+ n-2-2n_3-3n_4\,.
$$
\end{lem}

\begin{lem}\label{lemm1}
Let $T$ be a chemical tree with $n$ vertices. Then
\begin{eqnarray*}
s(T)&=&\frac{4(n-2)}{n}+\frac{n-2}{n}[2n_3(T)+4n_4(T)].
\end{eqnarray*}
\end{lem}

\begin{proof}
Since $|E(T)|=n-1$,
\begin{eqnarray*}
s(T)&=&n_1(T)\left[\frac{2(n-1)}{n}-1\right]+n_2(T)\left[2-\frac{2(n-1)}{n}\right]+n_3(T)\left[3-\frac{2(n-1)}{n}\right]\\
&+&n_4(T)\left[4-\frac{2(n-1)}{n}\right]\\
&=&\frac{n-2}{n} n_1(T)+\frac{2}{n} n_2(T)+\frac{n+2}{n} n_3(T)+\frac{2n+2}{n} n_4(T).
\end{eqnarray*}
We now apply Lemma \ref{lemm0} to deduce that $s(T) = \frac{4(n-2)}{n}+\frac{n-2}{n}[2n_3(T)+4n_4(T)]$, proving the lemma.
\end{proof}

\begin{cor}
Let $T$ be a chemical tree with $n$ vertices. Then $s(T)\geq \frac{4(n-2)}{n}$, with equality if and only if  $T\cong P_n$.
\end{cor}

\begin{lem}\label{lemm2}
Let $G$ be an $(n,1)$-chemical graph. Then $s(G) = 2n_3(G)+4n_4(G).$
\end{lem}

\begin{proof}
Since $|E(G)|=n$,
\begin{eqnarray*}
s(G)&=&n_1(G)[2-1]+n_2(G)[2-2]+n_3(G)[3-2]+n_4(G)[4-2]\\
&=&n_1(G)+n_3(G)+2n_4(G),
\end{eqnarray*}
and by Lemma \ref{lemm0}, $s(G) = 2n_3(G)+4n_4(G)$, as desired.
\end{proof}

\begin{thm}\label{lemm3}
Let $G$ be an $(n,c)$-chemical graph such that $c\geq 2$ .
\begin{enumerate}
\item If $n>2c-2$, then \begin{eqnarray*}
s(G)&=&\frac{1}{n}\Big[(2n-4c+4)n_3(G)+(4n-4c+4)n_4(G)\Big].
\end{eqnarray*}
\item If $n\leq 2c-2$, then  $s(G)=\frac{4(n-c+1)}{n}n_4(G)$.
\end{enumerate}
\end{thm}

\begin{proof}
By definition, $|E(G)|=n+c-1$.
\begin{enumerate}
\item $n>2c-2$. Then,
\begin{eqnarray*}
s(G)&=&n_1(G) \left[\frac{2(n+c-1)}{n}-1\right]+n_2(G)\left[\frac{2(n+c-1)}{n}-2\right]\\ &+& n_3(G)\left[3-\frac{2(n+c-1)}{n}\right] + n_4(G)\left[4-\frac{2(n+c-1)}{n}\right]\\
&=&\frac{n+2c-2}{n}n_1(G)+\frac{2c-2}{n}n_2(G)+\frac{n-2c+2}{n}n_3(G)+\frac{2n-2c+2}{n}n_4(G),
\end{eqnarray*}
and by Lemma \ref{lemm0}
\begin{eqnarray*}
s(G)&=&\frac{1}{n}\Big[(2n-4c+4)n_3(G)+(4n-4c+4)n_4(G)\Big].
\end{eqnarray*}

\item $n\leq 2c-2$. By the well-known result of Euler,  $2|E(G)|$ $=$ $\sum_{v\in V(G)}deg_G(v)\leq\sum_{v\in V(G)}4$ $=$ $4n$. Therefore, $c=|E(G)|-n+1\leq n+1$. Thus $2c-2\leq 2n$ and

 \begin{eqnarray*}
s(G)&=&n_1(G)\left[\frac{2(n+c-1)}{n}-1\right]+n_2(G)\left[\frac{2(n+c-1)}{n}-2\right]\\ &+& n_3(G)\left[\frac{2(n+c-1)}{n}-3\right] + n_4(G)\left[4-\frac{2(n+c-1)}{n}\right]\\
&=&\frac{n+2c-2}{n}n_1(G)+\frac{2c-2}{n}n_2(G)+\frac{2c-2-n}{n}n_3(G)+\frac{2n-2c+2}{n}n_4(G),
\end{eqnarray*}
and by Lemma \ref{lemm0}, $s(G)=\frac{4(n-c+1)}{n}n_4(G)$.
\end{enumerate}
Hence the result.
\end{proof}

\vskip 0.4 true cm

\noindent{\textbf{Acknowledgments.}} The research of the  second author was partially supported by the University of Kashan under grant no 364988/109.



\begin{thebibliography}{20}
\bibitem{1} M. O. Albertson, The irregularity of a graph, \textit{Ars Combin.} \textbf{46} (1997) 219--225.

\bibitem{15} A. Ali, E. Milovanovi\'c,  M. Mateji\'c and I. Milovanovi\'c, On the upper bounds for the degree deviation of graphs, \textit{J. Appl. Math. Comput.}, DOI: 10.1007/s12190-019-01279-6.

\bibitem{2} H. Abdo, S. Brandt and D. Dimitrov, The total irregularity of a graph, \textit{Discrete Math. Theor. Comput. Sci.} \textbf{16} (2014) 201--206.


\bibitem{2-1} A. Ghalavand, A. R. Ashrafi and I. Gutman, Extremal graphs for the second multiplicative Zagreb index, \textit{Bull. Int. Math. Virtual Inst.} \textbf{8} (2) (2018) 369--383.

\bibitem{3} J. A. de Oliveira, C. S. Oliveira, C. Justel and N. M. Maia de Abreu, Measures of irregularity of graphs, \textit{Pesq. Oper.} \textbf{33} (3) (2013) 383--398.

\bibitem{4} V. Nikiforov, Eigenvalues and degree deviation in graphs, \textit{Linear Algebra Appl.} \textbf{414} (2006) 347--360.
\end{thebibliography}
\end{document}